\documentclass{article}
\usepackage{amsmath,amsfonts}
\usepackage{amssymb}
\usepackage{algorithmic}
\usepackage{algorithm}
\usepackage{array}
\usepackage{textcomp}
\usepackage{stfloats}
\usepackage{url}
\usepackage{verbatim}
\usepackage{graphicx}
\usepackage{hyperref}
\usepackage{amsthm}
\usepackage{float}
\usepackage{cite}
\usepackage{mathtools}
\usepackage{color}
\usepackage{stackengine}
\mathtoolsset{showonlyrefs}

\usepackage{geometry}
\def\BibTeX{{\rm B\kern-.05em{\sc i\kern-.025em b}\kern-.08em
    T\kern-.1667em\lower.7ex\hbox{E}\kern-.125emX}}
\usepackage{balance}
\usepackage[caption=false, font=footnotesize]{subfig}

\DeclareMathOperator{\R}{\mathbb{R}}

\def \E{\mathbb{E}}

\def\1{{\bf 1}}

\def \N{\mathbb{N}}

\def\Nc{{\cal N}}
\def\argmin_#1{\underset{#1}{\mathrm{argmin\, }}}

\def \trans{^{\scriptscriptstyle{\intercal}}}

\newcommand{\dps}[1]{\displaystyle{#1}}

\def \trans{^{\scriptscriptstyle{\intercal}}}

\newtheorem{Rem}{Remark}
\newtheorem{Prop}{Proposition}

\title{The GroupMax neural network approximation of convex functions.}
\author{ Xavier \textsc{Warin}
\footnote{EDF R\&D \& FiME \sf \href{mailto:xavier.warin at edf.fr}{xavier.warin at edf.fr}} }

\begin{document}

\maketitle
\begin{abstract}
    We present a new neural network to approximate convex functions. This network has the particularity to approximate the function with cuts which is, for example, a necessary feature to approximate Bellman values when solving linear stochastic optimization problems.
    The network can be easily adapted to partial convexity.
    We give an universal approximation theorem in the full convex case and give many numerical results proving it efficiency. The network is competitive with the most efficient convexity preserving neural networks and can be used  to approximate functions in high dimension.
\end{abstract}

\section{Introduction}
Neural networks are an effective tools to approximate function  and numerically generally outperform  classical regression  using an expansion on a function basis. Some classical Universal Approximation theorem for neural network with bounded depth are given in \cite{cybenko,hornik} when the activation function is non polynomial.  A "dual" result is given in \cite{Gripenberg} where the number of  the hidden layers can be taken arbitrary large with bounded  widths if the activation function is non affine, continuous and twice continuously differentiable. Recently, \cite{lyons} slightly improves the results with only non affine, continuous and continuously differentiable activation functions.

The approximation of a convex function  has recently been theoretically investigated for example in \cite{gaubert2018} for a one-layer feedforward neural network with exponential activation functions in the inner layer and a logarithmic activation at the output.
Numerically this problem has been investigated in \cite{amos2017input} developing the Input Convex Neural Network (ICNN) methodology.
This approach is effective and has been widely used in many applications where the convexity of the approximation is required, for example in optimal transport problems \cite{makkuva2020optimal,korotin2019wasserstein}, in optimal control problems as in \cite{chen2018optimal,agrawal2020learning}, in inverse problems \cite{mukherjee2020learned}, or in general optimization problems \cite{chen2020input}  just to quote some of them.

In some cases, a  simple convex approximation of the function in not sufficient. For example, in multi stage  stochastic linear optimization, some methods such as  the Stochastic Dual Dynamic Programming  method  (SDDP) \cite{shapiro2011analysis} solve  transition problems starting from a state  using an  approximation by cuts of the Bellman function at the end of the transition period. This transition problem is  solved  using a Linear Programming solver.
The cut approximation of a convex function using regression methods  has already been investigated in  \cite{balazs2015near,ghosh2019max} or  \cite{ghosh2022} leading to some  max affine approximation.

This article proposes a different approach using a new network permitting to efficiently approximate a convex function by cuts using some ideas similar to the GroupSort network developed in \cite{anil2019sorting} and analyzed in \cite{tanielian2021approximating}.

In a first part, we present the network  supposing  that the function to approximate is convex with respect to  the input.  An universal approximation theorem is then given for this full convex case.
Then,  supposing that the function is only convex with respect to a part of the input, we extend our representation using conditional cuts.
In a second part some numerical examples are given showing that the network is clearly superior to a  network generating simple cuts and is competitive with  the ICNN method.
At last  a conclusion is given. 

\section{The GroupMax Network}
\noindent Let $f$ be a real valued convex function defined on $\R^d$.
We have the following representation \cite[Chapter 3.2.3]{boyd2004convex}: Define $\tilde f$ as:
\begin{align}
    \tilde f(x) =  \sup \{ g(x) | \, g(x) \mbox{ affine}, g(z) \le f(z)\}\label{eq:fConvSup}
\end{align}
then $\tilde f(x) = f(x)$ for $x \in \mbox{int Dom}(f)$.
In the sequel an approximation by cuts of a convex function will refer to a max-affine affine representation \eqref{eq:fConvSup}. A cut is one of the affine functions in this max-affine representation.

From equation \eqref{eq:fConvSup}, we may think of a first single layer network $h$:
\begin{equation}
    h^{\theta}(x)= \max_{i=1}^N A_i.x + b_i
    \label{eq:simpleCut}
\end{equation}
where $A_i \in \R^{d}$,  $b_i \in \R$, and $\theta = (A_i)_{i=1,N} \cup (B_i)_{i=1,N}$ is the set of parameters to optimize. In this single layer network, $N$ is the number of neurons and the $\max$ function is not applied componentwise but on the global vector. 
Then trying to approximate the convex function $f$ on  $\mathcal{D}$ we solve
\begin{equation}
 \theta^* = \argmin_{\theta} \E( (f(X) - h^\theta(X))^2) 
 \label{eq:minEq}
\end{equation}
where $X$ is a random variable  for example uniformly distributed  in $\mathcal{D}$. As we will see later, this first network is too simple and leads to bad approximations. We then develop a new network generating cuts with a given number of layers.
\subsection{A Network for Fully Convex Functions}
\label{par:fullConv}
\subsubsection{The Network with q Layers}
\noindent We impose to simplify that $M$, the number of neurons, is kept constant for all layers and we  define the group size $G$ as in \cite{anil2019sorting} such that $K=\frac{M}{G}$ is an integer corresponding to  the number of groups.
For $x \in \R^d$, the network is defined by recurrence as
\begin{align}
    z^{1}= & \rho( A^1 x+ B^1)  \\
   z^{i}= & \rho( (A^i)^{+} z^{i-1}+ B^i)  1 < i < q \\
    h^\theta(x)=  & \hat \rho(  (A^q)^{+} z^{q-1} +B^q) 
 \label{eq:convexNet}
\end{align}
where $q$ is the number of layers, $y^+ = \max(y,0)$, $A^1 \in \R^{M,d}$, $B^1 \in \R^{M}$, $A^j \in \R^{M \times K}$, $B^j \in \R^{M}$, $j=2, \dots, q$, defining  the set of parameters as $\theta = (A^j)_{j=1,q} \cup (B^j)_{j=1,q}$.

The activation function $\rho$ and $\tilde \rho$ are defined as follows:
\begin{itemize}
    \item $\hat \rho$ is a $\R$ valued function defined on $\R^M$ where 
    \begin{equation}
    \hat \rho (x) = \max(x_1, \dots, x_M) \mbox{ for } x \in \R^M
\end{equation}
\item $\rho$ is a function   from $\R^M$ in $\R^K$ such that 
\begin{align*}
  \rho(x)_i = \max (x_{1+(i-1)G}, \dots, x_{iG}) \mbox{ for }  i=1, \dots, K
\end{align*}
\end{itemize}
and an example of the structure of the network given in Fig. \ref{fig:gMaxStruc}. \\
\begin{figure}[!t]
    \centering
    \includegraphics[width=\linewidth]{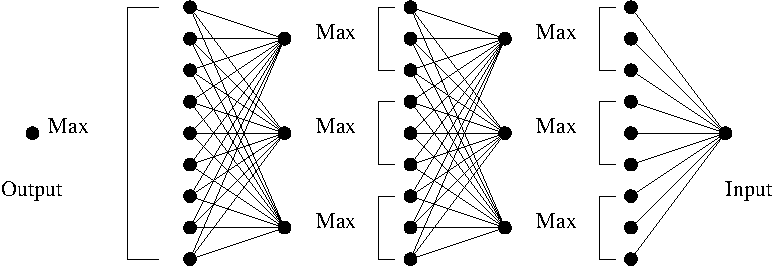}
    \caption{Network structure  in dimension 1, group size $G=3$, $K=3$ groups and 3 layers of $M=9$ neurons. }
    \label{fig:gMaxStruc}
\end{figure}
This network gives an approximation of $f$ by some cuts:
clearly by positivity of the $(A_i)^{+}$, we get that
\begin{flalign}
h^\theta(x) = & \max_{i_q \in [1,M]} \max_{ \scriptsize \begin{array}{c}
     i_j  \in [1,G ], \\
      j=1, \dots, q-1
\end{array}} \sum_{\scriptsize \begin{array}{c} k_j =1\\  j=1,q-1 \end{array}}^K  (A^q)^{+}_{i_q, k_{q-1}} \\
& \prod_{n=1}^{q-2} (A^{n+1})^{+}_{i_{n+1}+(k_{n+1}-1)G, k_{n}}  \sum_{m=1}^d A^1_{i_1+(k_1-1)G,m} x_m \\
& +C 
\label{eq:cutFor}
\end{flalign}
where $C$ is function of the  $B^i$ , $i=1,\dots,q$ and $A^i$, $i=1, \dots,q-1$. Then the number of cuts can be calculated as  $ M G^{K(q-1)}$.
\begin{Rem}
Using $M$ neurons on the first layer and $\tilde{M} \ge \frac{M}{G}$ neurons on the following layers, we notice that if we take  $B^j=0$ for $1< j \le q$, and $A^j$ null except $(A^j)_{i,i}=1$, for $i=1, \dots, \frac{M}{G}$ for $1 < j \le q$, then the cuts generated by the one layer network are the same as the cuts generated by the $q>1$ layers network. Then cuts generated by the network \eqref{eq:simpleCut}  can be reached with the network \eqref{eq:convexNet}.
\label{remark1}
\end{Rem}
\begin{Rem}
Using equation \eqref{eq:cutFor}, it is possible to reconstruct the underlying cuts which is, for example, necessary to solve  stochastic linear optimization  problems using Benders cuts.
\end{Rem}

\subsubsection{Universal Approximation Theorem}
We denote $\mathcal{N}(M_1,\dots,M_q,K)$ the set of generated functions $h^{\theta_{M_1, \dots,M_q,K}}$ by the network \eqref{eq:convexNet} with $\theta_{M_1,..,M_q,K} \in \R^{M_1 (d+1) + \sum_{i=2}^q M_i (K+1)}$ where the number of neurons for layer $q$ is $M_q$. We introduce the space of functions generated letting the number of neurons vary:
\begin{align}
    \mathcal{N}(K,q) = \cup_{(M_1,\dots, M_q) \in \N^q} \mathcal{N}(M_1,\dots,M_q,K)
    \label{eq:GroupMaxConv}
\end{align}
\begin{Prop}
Let $f$ be a convex function on $\R^d$, then $\mathcal{N}(K,q)$ approximates arbitrarily well $f$ by below on every compact $\hat K$ for the sup norm.
\end{Prop}
\begin{proof}
Due to remark \ref{remark1}, it is sufficient to prove that for a given function $f$, for each $\epsilon$,  there exists a finite number of cuts $(A_i,B_i)_{i=1,M}$   such that $g(x)= \displaystyle{\max_{i=1,M}} A_i x  + B_i$ and such that $g(x) \le f(x)$  and $\displaystyle{\sup_{x \in \hat K}} f(x) -g(x) \le \epsilon$. 

Let suppose the converse. There exist $\epsilon$, such that for  $n_0$ being chosen arbitrary, and whatever the cuts we take generating a function $g^{0}$ below $f$, then there exists $x_0$ such that $f(x_0) -g^0(x_0) > \epsilon$.
Using \cite[Proposition A]{guessab2004convexity}, we can generate a cut $(A,B)$ such that $f(x_0)= A x_0 +b$ and $f(x) \ge Ax +B$ on $\hat K$. Let use define, $g^{1}(x) = max( g^{0}(x), A x+B)$. Due to the hypothesis, there there exist $x_1$ such that $f(x_1) -g^1(x_1) > \epsilon$ and  we can build  a sequence $(x_i, g^i)$, $i \ge 0$ such that $g^i$ is below $f$ and  $ f(x_i) -g^j(x_i) \ge \epsilon$  for all $j \le i$ and $ f(x_i) = g^{j}(x_i)$ for $j > i$.

We can extract a sequence $\tilde x^i$ from $(x_i)_{i \ge 0}$ that converges to $\tilde x \in \hat K$. As $f$ is convex on $\R^d$ it is continuous on  $\hat K$.
Then letting $i$ go to infinity in 
$ f(\tilde x_i) -g^j(\tilde x_i) \ge \epsilon$ for $j$ fixed, and using $g^j$ continuity we get $f(\tilde x) -g^j(\tilde x) \ge \epsilon$ for all $j$.  Defining $ g = \displaystyle{\sup_{i >0} } \; g_i$ which is convex so continuous, we get that $f(\tilde x) -g(\tilde x) \ge \epsilon$.

Starting from $ f(x_i) = g^{j}(x_i)$ for $j > i$ and first letting $j$ to infinity, we get $ f(x_i) = g(x_i)$. Letting $i$ go to infinity and using the continuity of $g$ and $f$ we get the contradiction.
\end{proof}

\subsection{An Extension for Partial Convex Functions}
 \noindent We suppose that  $ x= ( \tilde x,y) \in \R_{d}$ where we have convexity in $y \in \R^{k}$,
We simply modify our network as done in   \cite{amos2017input} to take into account  the non convexity in $\tilde x$:  other networks tested could not outperform the one  proposed  in \cite{amos2017input}.
The recursion is given by:
\begin{flalign}
u_0 =& \tilde x ,  \quad z_0= 0\nonumber \\
u_{i+1} = &  \tilde \rho( \tilde W_{i+1} u_i + \tilde b_{i+1})  \nonumber \\ 
z_{i+1} = & \rho( [W^{(z)}_{i+1} \otimes (W^{(zu)}_{i+1} u_i + b^{(z)}_{i+1}) ]^+  z_i   +  \nonumber\\ & W^{(y)}_{i+1}( y \circ (W^{(yu)}_{i+1} u_i + b^{(y)}_{i+1})) + W^{(u)}_{i+1} u_i + b_{i+1}) \\
&  \mbox{ for } i <  q-1 \nonumber \\
h^{\theta}(\tilde x,y)= &  \hat\rho( [W^{(z)}_q \otimes (W^{(zu)}_q u_{q-1} + b^{(z)}_q) ]^+  z_q   + \nonumber \\ & W^{(y)}_q( y \circ (W^{(yu)}_q u_{q-1} + b^{(y)}_q)) + W^{(u)}_q u_{q-1} + b_q)
\label{eq:GroupMax}
\end{flalign}
where $\circ$ denote the Hadamard product, $\otimes$ is applied between a matrix $A \in \R^{m \times n}$ and a vector $B \in \R^n$ such that $A \otimes B \in \R^{m \times n}$ and $(A \otimes B)_{i,j}=  A_{i,j} B_j$. 

In recursion \eqref{eq:GroupMax}, $\tilde \rho$ is a classical activation function such as the ReLU function.
Using $m_x$ neurons for the non convex part of  the function and $m_y$ neural networks for the convex part, $\tilde W_0 \in \R^{m_x \times (d-k)}$, $\tilde W_i \in \R^{m_x \times m_x}$ for $i>0$,  $W^{(zu)}_i \in \R^{m_y \times m_x}$ for $i > 0$, $W^{(z)}_i \in \R^{m_y \times m_y}$ for $i \le q$. We do not detail  the size of the different matrices $W^{(y)}$, $W^{(yu)}$, $W^{(u)}$  and the different biases that are obvious.  $\theta$ is the set of all these weights and bias.
Introducing 
\begin{align*}
    A^i(\tilde x)= & [W^{(z)}_i \otimes (W^{(zu)}_i u_{i-1} + b^{(z)}_i) ]^+ \\
    \tilde{A}^i(\tilde x) = &  W^{(y)}_i \otimes( W^{(yu)}_i u_{i-1} + b^{(y)}_i)\\
    B^i(\tilde x)= & W^{(u)}_i u_{i-1} + b_i
\end{align*}
where we can note by recurrence that $u_i$ is a  non linear function of $ \tilde x$, equation \eqref{eq:GroupMax} can be rewritten as
\begin{align*}
    (z_{i+1})_j= &  \max_{ l \in [1,G]}  [A^{i+1}(\tilde x) z_i + \tilde A^{i+1}(\tilde x) y +  B^{i+1}(\tilde x)]_{(j-1)G+l} \\
    & \quad  \mbox{ for  }  j=1, \dots, K, \mbox{ and  } i < q-1 \\ 
    h^\theta(\tilde x,y)= &\max_{l \in [1, m_y]} [A^q(\tilde x) z_{q-1} + \tilde A^q(\tilde x) y + B^q(\tilde x)]_{l}
\end{align*}
Similarly to section \ref{par:fullConv} ,
\begin{align}
    h^\theta( \tilde x,y)= \max_{l=1, \dots, m_y G^{K  (q-1) }}[A(\tilde x) y + B(\tilde x)
    \label{eq:GroupMaxASCut}
\end{align}
where $A(\tilde x) \in \R^{m_y G^{ K(q-1) } \times k}$ is a function of the $(A^i(\tilde  x))_{i=1,\dots,q}$ and $(\tilde A^i(\tilde  x))_{i=1,\dots,q}$ matrices. Similarly  $B(\tilde  x) \in \R^{m_y G^{ K (q-1)}}$ is a function of  the $(B^i( \tilde x))_{i=1,\dots,q}$, $(A^i(\tilde x))_{i=1,\dots,q-1}$ and $(\tilde A^i(\tilde x))_{i=1,\dots,q-1}$. Then it is clear that this recursion permits to define some cuts conditional to $ \tilde x$.

\section{Numerical Results}
\noindent In all numerical results, we use tensorflow \cite{abadi2016tensorflow} with an ADAM gradient descent algorithm \cite{kingma2014adam}.
\subsection{Some One Dimensional Results}
\noindent We consider the convex function $f(x) = f_i(x)$ for case $i=1,\dots,5$ where
\begin{enumerate}
   \item $f_1(x)= x^2 $ 
   \item $f_2(x)= x^2 +  10 [(e^x-1) 1_{x<0} +x 1_{x\ge 0}]$ 
     \item $f_3(x) = (|x|^2+1)^2$
      \item $f_4(x) = |x| 1_{|x| \le 3} + \frac{x^2-3}{2}$
\end{enumerate}
We regress  $f(x) + \epsilon$  with respect to $x$ where $\epsilon \sim \Nc(0,1)$ using 20000 gradient iterations, a batch size equal to $300$ and a learning rate equal to $1e-3$.
We then solve equation \eqref{eq:minEq} using $X \sim \mathcal{N}(0,4)$.
On Fig. \ref{fig:simpleCut1}, we see that the simple approximation \eqref{eq:simpleCut} gives visually not very good results and that the solution does not improve as we increase the number of cuts.
\begin{figure}[!t]
\centering
  \subfloat[$f_1$ ]{%
       \includegraphics[width=0.45\linewidth]{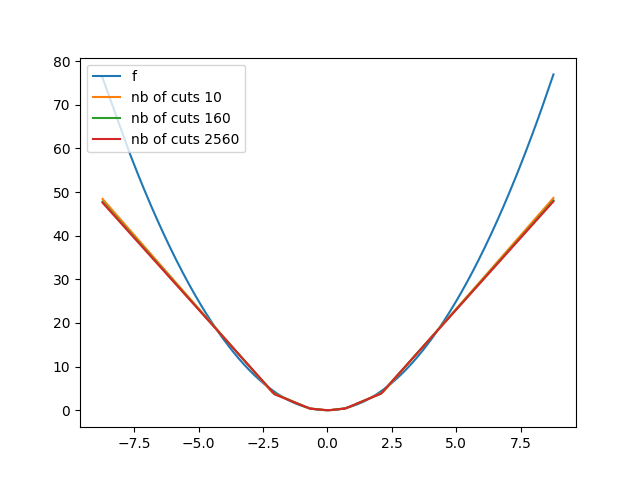}}
  \subfloat[$f_2$]{%
        \includegraphics[width=0.45\linewidth]{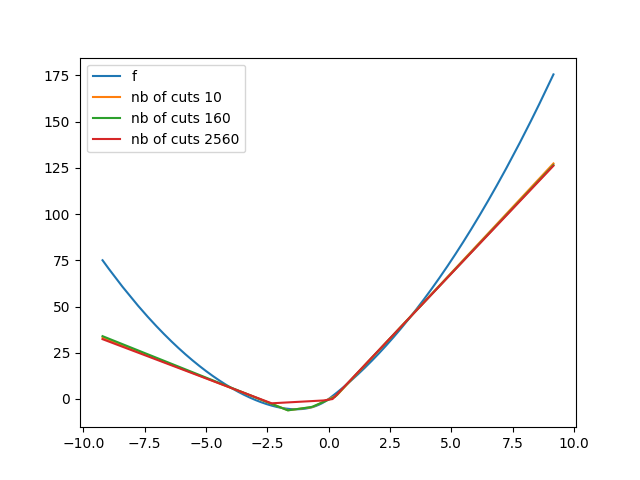}}
        \hfill
  \subfloat[$f_3$ ]{%
       \includegraphics[width=0.45\linewidth]{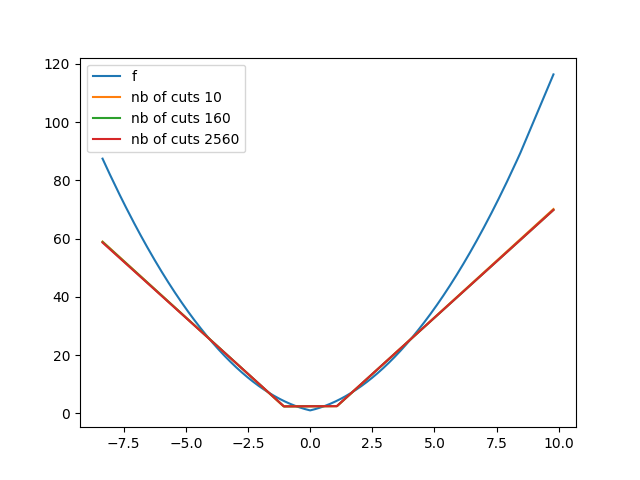}}
  \subfloat[$f_4$]{%
        \includegraphics[width=0.45\linewidth]{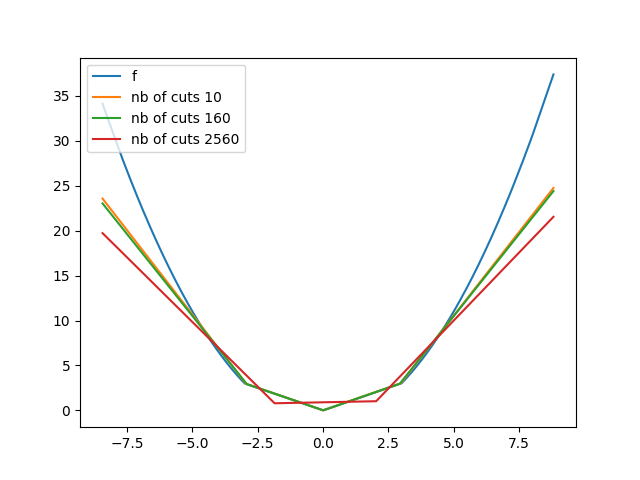}}
    \caption{ Solution of equation \eqref{eq:minEq} using network \eqref{eq:simpleCut} with different number of cuts. \label{fig:simpleCut1}}
\end{figure}
It was however possible to get some rather accurate solutions except in the tails  using network \eqref{eq:simpleCut} by renormalizing {\bf  both the input $x$ and the output $f(x)+ \epsilon$} such that both are centered with a unit standard deviation. Results are given on Fig. \ref{fig:simpleCut1Renorm}.
\begin{figure}[!t]
  \centering
  \subfloat[$f_1$]{%
       \includegraphics[width=0.45\linewidth]{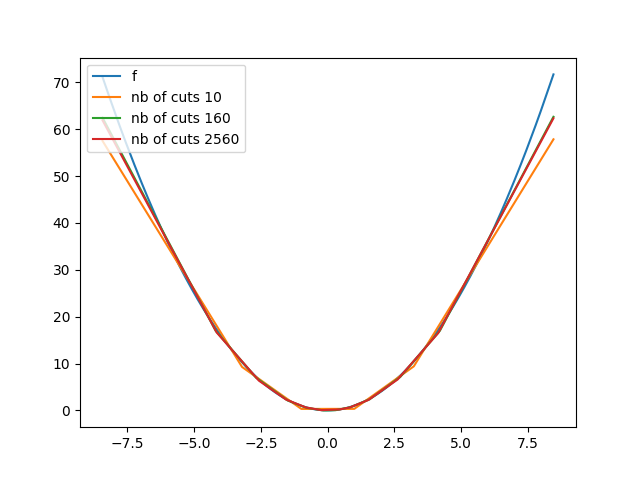}}
  \subfloat[$f_2$]{%
        \includegraphics[width=0.45\linewidth]{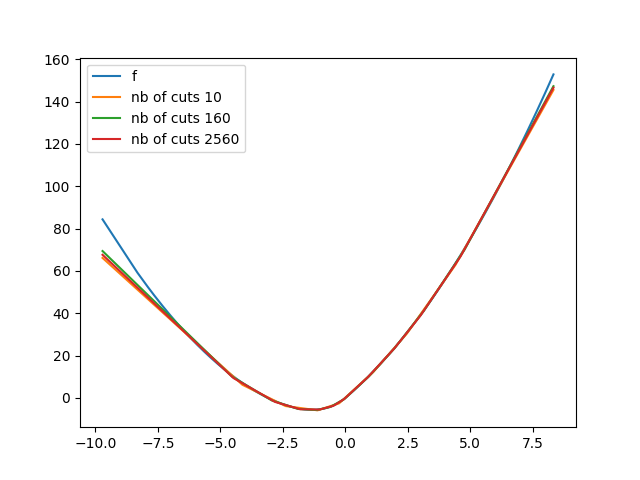}}
        \hfill
  \subfloat[$f_3$]{%
       \includegraphics[width=0.45\linewidth]{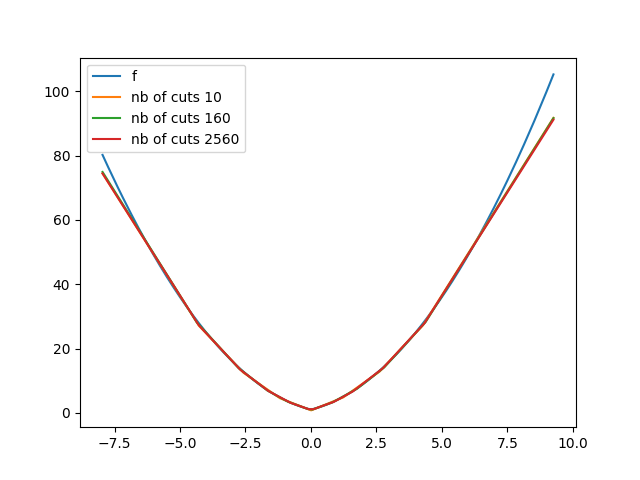}}
  \subfloat[$f_4$]{%
        \includegraphics[width=0.45\linewidth]{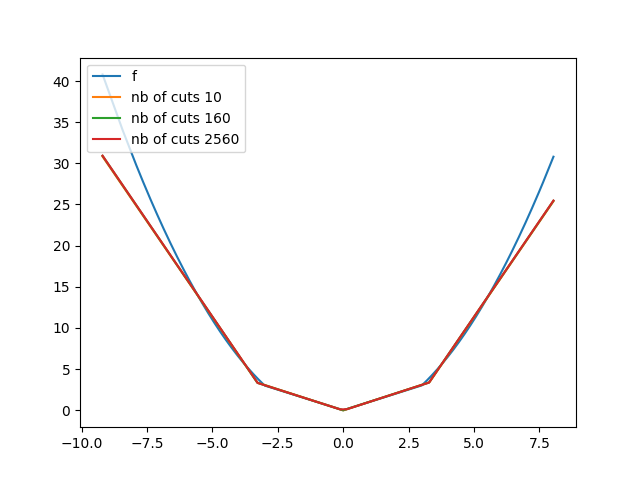}}
    \caption{ Solution of equation \eqref{eq:minEq} using network \eqref{eq:simpleCut} with different number of cuts using some input and output renormalization.\label{fig:simpleCut1Renorm}}
\end{figure}
On the Fig.  \ref{fig:comp1D} we plot the solution obtained using the GroupMax network, \cite{amos2017input} network and the feedforward network.  
For the GroupMax we use 3 layers with 10 neurons on each layer and the  group size $G=5$. As for the two other networks, we use 3 hidden layers with $10$ neurons and the ReLU activation function. 50000 gradient descent iterations are used.
\begin{figure}[!t]
  \centering
  \subfloat[$f_1$]{%
       \includegraphics[width=0.45\linewidth]{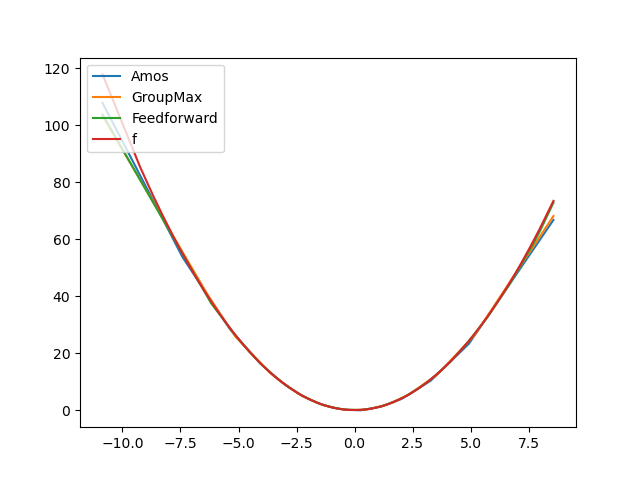}}
  \subfloat[$f_2$]{%
        \includegraphics[width=0.45\linewidth]{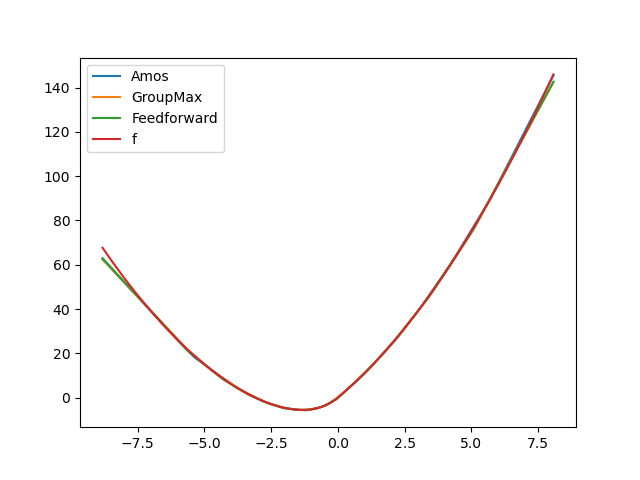}}
        \hfill
  \subfloat[$f_3$]{%
       \includegraphics[width=0.45\linewidth]{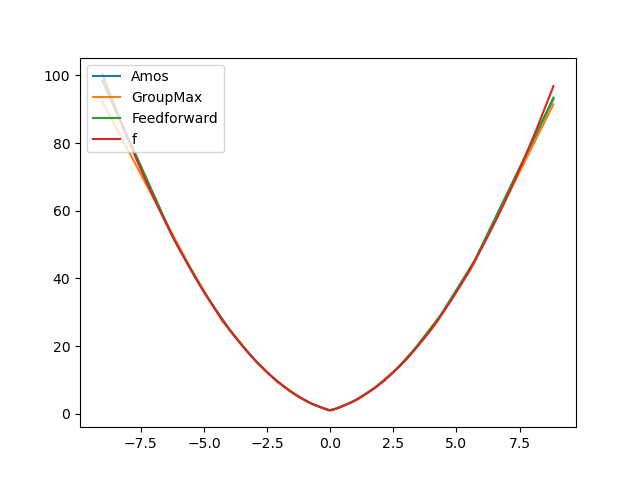}}
  \subfloat[$f_4$]{%
        \includegraphics[width=0.45\linewidth]{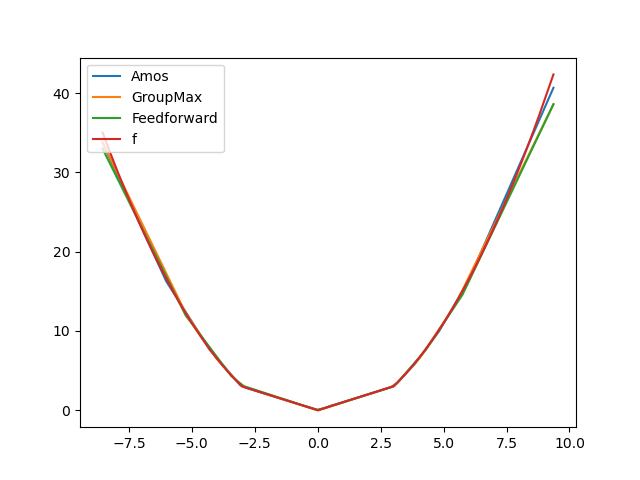}}
    \caption{ Solution of equation \eqref{eq:minEq} using the GroupMax,  the ICCN network \cite{amos2017input} and a feedforward network. \label{fig:comp1D}}
\end{figure}
The cuts generated on the previous examples by the GroupMax network  are given on Fig. \ref{fig:comp1DCut}.
\begin{figure}[!t]
  \centering
  \subfloat[$f_1$]{%
       \includegraphics[width=0.45\linewidth]{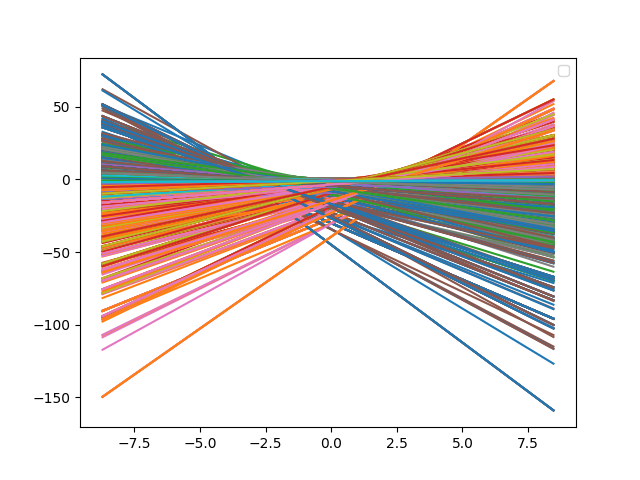}}
  \subfloat[$f_2$]{%
        \includegraphics[width=0.45\linewidth]{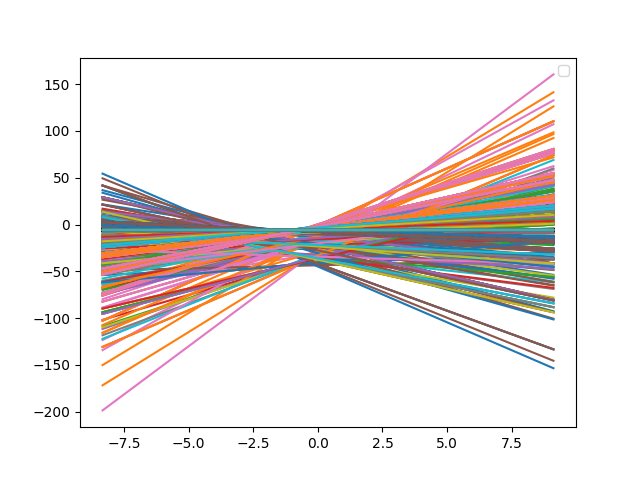}}
        \hfill
  \subfloat[$f_3$]{%
       \includegraphics[width=0.45\linewidth]{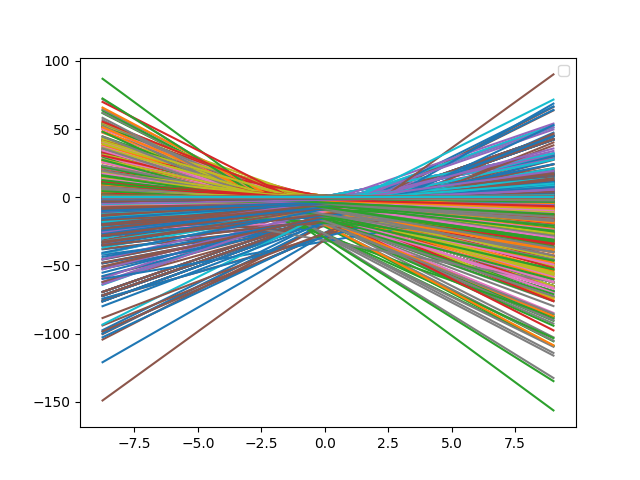}}
  \subfloat[$f_4$]{%
        \includegraphics[width=0.45\linewidth]{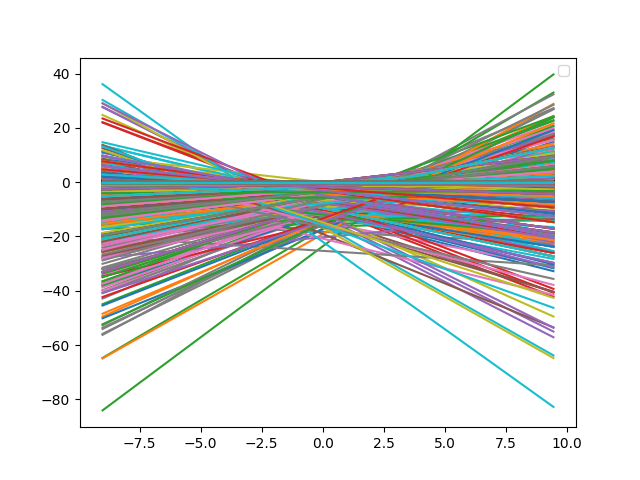}}
    \caption{ Cuts generated by the GroupMax network estimating solution of equation \eqref{eq:minEq} using $3$ layers and a group size of $5$.\label{fig:comp1DCut}}
\end{figure}
To see the accuracy of the GroupMax network as an interpolator, we optimize equation \eqref{eq:minEq}  trying to fit directly the function $f$ without any noise. Table \ref{tab:interp1D} gives the best MSE on $10$ test runs using Monte Carlo (1e6 samples). The solution obtained by the feedforward seems to be more accurate with the chosen parameters. Between \cite{amos2017input}, and the GroupMax it is hard to say which is the best.
\begin{table}
    \begin{center}
     \caption{ MSE for the Different Networks.}
    \label{tab:interp1D}
    \begin{tabular}{|c|c|c|c|c|}  \hline
  Network       & $f_1$ & $f_2$ & $f_3$ & $f_4$ \\ \hline
   Feedfoward      &  0.0004 &      0.0031  &  0.0004&  0.0001  \\ \hline
   \cite{amos2017input} network    &  0.0013 &  0.0050&  0.0019&  0.0006\\ \hline
  GroupMax network      &  0.0012  & 0.0014 &  0.0029 &  0.0002\\ \hline
    \end{tabular}
    \end{center}
\end{table}
In table \ref{tab:KSway1D}, we give the results obtained (best of $10$ runs) for different values of parameter $K$ using $12$ neurons per layer. As before the best of 10 runs is kept. 
\begin{table}
    \centering
        \caption{Influence of the Group Size on the MSE.}
    \label{tab:KSway1D}
   \begin{tabular}{|c|c|c|c|c|}  \hline
  $K$       & $f_1$ & $f_2$ & $f_3$ & $f_4$ \\ \hline
  2        & 0.0010  &  0.0015 &  0.0019 & 0.00028 \\ \hline
  4        & 0.0007  &  0.0012 & 0.0014  & 0.00023\\ \hline
6        & 0.0010 &  0.0018 & 0.0032  & 0.00038 \\ \hline
12        & 0.00071& 0.0071  & 0.0137  & 0.00085 \\ \hline
    \end{tabular}
\end{table}
The results seem to indicate that the group size should remain rather low.
\begin{table}
    \centering 
    \caption{Influence of the Number of Layers q on the MSE.}
    \label{tab:NLSway1D}
   \begin{tabular}{|c|c|c|c|c|}  \hline
  q       & $f_1$ & $f_2$ & $f_3$ & $f_4$ \\ \hline
  2        & 0.0073  &  0.0054 & 0.062  & 6.7e-4 \\ \hline
  3        & 0.0017  &  0.0010 & 0.021   &  5.9e-4\\ \hline
4        & 7.5e-4 &  0.0011 & 0.001  &  9.5e-5 \\ \hline
5        & 2.7e-4&  7e-4  & 4e-4  &  3.5e-5\\ \hline
    \end{tabular}
\end{table}
At last, the influence of the number of layers $q$ is given in table \ref{tab:NLSway1D} taking 12 neurons per layer and a group size of $2$ in the GroupMax network.  The best of $10$ runs is given. The accuracy clearly improves as we increase the number of layers.

\subsection{Testing Partial Convexity in 2D}
\noindent We suppose that we want to interpolate a function  which is convex only in its second dimension and test the error obtained 
in solving \eqref{eq:minEq} in two cases.
In the first case, we suppose that $X \sim \mathcal{N}(0,1)^2$ , then that $X \sim  \mathbf{U}([-2,2]^2)$.
We test the following functions convex in $y$:
\begin{enumerate}
   \item $f_5(x,y)= y^2 |x+2x^3|$ 
   \item $f_6(x,y)=  (1+|y|) |x+2x^3|$
   \item $f_7(x,y)=  y^{+} |x| + x^2$
\end{enumerate}
We keep the same parameters as before for the different networks. For the ICNN and the GroupMax network we take the same number of neurons both for the convex and non convex part.
We first keep a learning rate equal to $10^{-3}$ and a batch size equal to $300$.
We take $50000$ gradient iterations. Results are given in table \ref{tab:convNonConvGaussian} sampling $X \sim \mathcal{N}(0,1)^2$ and \ref{tab:convNonConvUniform} with $X \sim  \mathbf{U}([-2,2]^2)$. Best of 10 runs are given. Surprisingly, the feedforward network behaves very badly especially sampling a Gaussian law.

\begin{table}
    \centering
        \caption{MSE with  $X \sim \mathcal{N}(0,1)^2$.}
    \label{tab:convNonConvGaussian}
    \begin{tabular}{|c|c|c|c|} \hline
       Network       & $f_5$ & $f_6$ & $f_7$  \\ \hline
      Feedforward   &  1.6 &  1.99  &  1.7e-3  \\ \hline
      \cite{amos2017input}   & 0.019  & 0.073   &  6.5e-6 \\ \hline
      GroupMax  &  0.057  &  0.21  &  8e-7  \\ \hline
    \end{tabular}
\end{table}
 \begin{table}
    \centering
     \caption{MSE with $X \sim  \mathbf{U}([-2,2]^2)$.}
    \label{tab:convNonConvUniform}
    \begin{tabular}{|c|c|c|c|} \hline
       Network       & $f_5$ & $f_6$ & $f_7$  \\ \hline
      Feedforward   &  0.01& 0.13   & 1.8e-4    \\ \hline
      \cite{amos2017input}   & 5.8e-4  & 2.7e-3   &  5e-7  \\ \hline
      GroupMax  & 3.5e-3   & 6e-3   &  4e-7  \\ \hline
    \end{tabular}
\end{table}
\begin{Rem}
It was possible to get better results for the $f_6$ function using the feedforward network using up to $80$ neurons or increasing the number of layers. The results obtained  were not as good as the ones obtained by the other networks.
\end{Rem}
We test the GroupMax network on the two cases with $12$ neurons on each layer, with a group size $K$ equal to 3, with different numbers of layers $q$ on table \ref{tab:NLSway1DNonConGauss} \ref{tab:NLSway1DNonConUnif}. The best of $10$ runs is given. Sampling an uniform law,  the error clearly decreases with the number of layers while sampling a gaussian law this decrease is not observed.
\begin{table}
    \centering
     \caption{Influence on the MSE  of the Number of Layers q Sampling $X \sim \mathcal{N}(0,1)^2$.}
    \label{tab:NLSway1DNonConGauss}
   \begin{tabular}{|c|c|c|c|}  \hline
  q       & $f_5$ & $f_6$  & $f_7$\\ \hline
  3        & 0.052 & 0.088  & 7e-7\\ \hline
  4        & 0.033 & 0.099 & 3e-6\\ \hline
  5        & 0.068 & 0.051 & 3e-6 \\ \hline
    \end{tabular}
\end{table}
\begin{table}
    \centering
        \caption{Influence on the MSE of the Number of Layers q Sampling  $X \sim  \mathbf{U}([-2,2]^2)$.}
    \label{tab:NLSway1DNonConUnif}
   \begin{tabular}{|c|c|c|c|}  \hline
  q       & $f_5$ & $f_6$  & $f_7$\\ \hline
  3        & 2.3e-3 & 4.8e-3  &  1.9e-7 \\ \hline
  4        & 9.6e-4 & 3.9e-3  &  5e-7\\ \hline
  5        & 5.9e-4 &  1.8e-3  & 5e-7 \\ \hline
    \end{tabular}
\end{table}

 \subsection{Convexity in Higher Dimension}
 \noindent We try to interpolate here a function in higher dimension.
 We take two cases:
 \begin{enumerate}
     \item First, it is the square of the $L_2$ norm.
     \begin{align}
         f_8(x) = ||x||^2
     \end{align}
     \item For the second we generate randomly  $A$ a positive definite matrix in dimension $d$ and take
     \begin{flalign}
     f_9(x)= \dps{\sum_{i=1}^d} (|x_i| + |1-x_i|) + x\trans A x
     \end{flalign}
 \end{enumerate}
 For the network \cite{amos2017input}, we keep on using 3 hidden layers with 10 neurons. As for the feedforward network we also take 3 layers with 10 neurons.
 For the GroupMax we take 5 layers of 10 neurons with a group size of 2. The learning rate is still equal to $10^{-3}$ and we take $100000$ gradient iterations.
 \begin{table}
     \centering
             \caption{MSE for $f_8$  Sampling $X \sim \mathcal{N}(0,1)^d$.
     \label{tab:convNDGaussianF1}}
     \begin{tabular}{|c|c|c|c|c|} \hline
      Dimension ($d$)    &  2 & 3 & 4 & 5  \\ \hline
      Feedforward     & 8e-4  & 1e-3 &  5e-3 & 0.016 \\ \hline
      \cite{amos2017input}   &  2.8e-3 & 1.2e-2 &6.9e-2   & 0.197 \\ \hline
     GroupMax   & 5.3e-4 & 5.1e-3  & 1.3e-2 &  0.030\\ \hline
     \end{tabular}
 \end{table}
 \begin{table}
     \centering
        \caption{MSE for $f_9$  Sampling $X \sim \mathcal{N}(0,1)^d$.
     \label{tab:convNDGaussianF2}}
     \begin{tabular}{|c|c|c|c|c|} \hline
      Dimension    &  2 & 3 & 4 & 5  \\ \hline
      Feedforward     & 1.2e-3 &  6.7e-3 & 0.026 &  0.105 \\ \hline
      \cite{amos2017input}   & 6e-3 & 3.2e-2 & 0.141 & 0.361 \\ \hline
     GroupMax   & 3.9e-3 & 2.8e-2 &  0.075&  0.228 \\ \hline
     \end{tabular}

 \end{table}
\begin{table}
     \centering
         \caption{MSE for $f_8$  Sampling $X \sim  \mathbf{U}([-2,2]^d)$.
     \label{tab:convNDGaussUniformF1}}
     \begin{tabular}{|c|c|c|c|c|} \hline
      Dimension    &  2 & 3 & 4 & 5  \\ \hline
       Feedforward     &  2.7e-4 & 7.7e-4& 1.6e-3 & 5.6e-3 \\ \hline
      \cite{amos2017input}   &  7.5e-4 & 5.5e-3 & 1.8e-2 & 4.5e-2 \\ \hline
     GroupMax   & 3.9e-4 & 1.5e-3 & 4.4e-3 & 1.2e-2 \\ \hline
     \end{tabular}
 \end{table}
 \begin{table}
     \centering
         \caption{MSE for  $f_9$  Sampling $X \sim  \mathbf{U}([-2,2]^d)$.
     \label{tab:convNDGaussUniformF2}}
     \begin{tabular}{|c|c|c|c|c|} \hline
      Dimension    &  2 & 3 & 4 & 5  \\ \hline
       Feedforward     &  7.8e-4  & 9.2e-3 & 1.2e-2 & 7.6e-2 \\ \hline
      \cite{amos2017input}   & 2.6e-3 & 2.3e-2 & 8.7e-2 & 0.38 \\ \hline
     GroupMax   & 2e-3 & 2.8e-2 &  6.3e-2&  0.17\\ \hline
      \end{tabular}
 \end{table}
 Results depending on the dimension of the problem are given in table \ref{tab:convNDGaussianF1},\ref{tab:convNDGaussianF2} ,\ref{tab:convNDGaussUniformF1}, \ref{tab:convNDGaussUniformF2} either sampling using a gaussian law for $X$ in equation \eqref{eq:minEq}, or sampling $X$ uniformly in $[-2,2]^d$. The best of $10$ runs is given. As an $L_2$ interpolator, classical feedforward seems to be the best and the GroupMax slightly  outperforms the \cite{amos2017input} network. Notice that we did not try other activation functions for the \cite{amos2017input} network and did not play with the number of neurons to upgrade the results.
 On the same cases we then increase the number of layers in dimension $5$ and give the results in table \ref{tab:dim5} using the GroupMax network  depending on the number of layers $q$. The best of $10$ runs is given. Clearly we see that this increase of the number of layers improves the results even if the approximation of $f_9$ remains not very good.\\
 \begin{table}
     \centering
         \caption{MSE for $f_8$ and $f_9$   Depending on the Number of Layers $q$ Sampling $X \sim  \mathbf{U}([-2,2]^5)$. }
     \label{tab:dim5}
     \begin{tabular}{|c|c|c|c|c|} \hline
       q   &  4  &  6 & 7  &  8 \\  \hline
        $f_8$  &  0.021 &   0.012 & 8.1e-3  &  7.7e-3 \\  \hline
        $f_9$ & 0.278  & 0.164  & 0.155   &  0.120 \\  \hline
     \end{tabular}
 \end{table}
 At last  we try to approximate a function in very high dimension :
     \begin{align}
         f_{10}(x,y) = -\frac{1}{2n} x\trans  x + \frac{1}{2m} y  \trans y 
     \end{align}
     where $x \in \R^n$ and $y \in \R^m$, with $n=376$, $m=17$.
     We train the different methods to approximate $f_{10}$ using   $X\sim (\mathcal{N}(0,1))^n$ and $Y \sim (\mathcal{N}(0,1))^m$. We estimate the MSE for each method in table \ref{tab:funcReviewer} depending on the number of layers $q$ taking $10$ neurons for \cite{amos2017input} network, $20$ neurons for the feedforward and $12$ neurons with a group size of $2$ for the GroupMax network. As previously, we take the best result of ten runs.
     The MSE obtained is small for all methods.
     The feedforward network's results are nearly independent of the number of layers while the two other methods get better results as the number of layers increases and tend to outperform the feedforward network.
     
     \begin{table}
     \begin{center}
         \caption{MSE for Function $f_{10}$ Depending on the Number of Layers $q$. }
     \label{tab:funcReviewer}
     \begin{tabular}{|c|c|c|c|} \hline
        q  &  Feed forward & \cite{amos2017input} & GroupMax \\ \hline
        3 & 0.0025  & 0.0041  &   0.0073 \\ \hline
        5 & 0.0025  & 0.0029   &  0.0037  \\ \hline
        7 & 0.0024  &  0.0022  &  0.0024   \\ \hline
        9 & 0.0025  &  0.0019   &  0.0019    \\ \hline
        \end{tabular}
     \end{center}
     \end{table}

 \section{Conclusion}
\noindent  A new effective network has been developed to approximate convex function or partially convex functions by cuts or conditional cuts. This network gives similar results to the best networks developed giving a convex or partially convex solution. This approximation by cuts can be used in many application where convexity of the approximation is required or could be used in multistage linear continuous stochastic optimization where of the Bellman values by cuts is necessary to use a Linear Programming solver.

\section{Acknowledgement}
\noindent We thank Vincent Lemaire and Gilles Pagès for helpful discussions.


 \end{document}